\PassOptionsToPackage{table,dvipsnames}{xcolor}
\documentclass{ifacconf}

\usepackage[pdftex]{graphicx}
    \graphicspath{{Figures/}}                   
    \DeclareGraphicsExtensions{.pdf,.jpg,.png}  
    \usepackage{svg}                            
    \usepackage{float}                          

\usepackage{amsmath}                            
\usepackage{amssymb}                            

\usepackage{mathtools}                          
\usepackage{breqn}                              
\usepackage{array}                              
\usepackage{nicefrac}
\usepackage{blkarray, bigstrut}                 

\usepackage{url}

\usepackage{dutchcal}
\usepackage{natbib}                             
\usepackage{siunitx}                            
\usepackage{textcomp}                           
\usepackage{algorithm}
\usepackage{algpseudocode}

\newtheorem{ass}{Assumption}
\newtheorem{remark}{Remark}
\newtheorem{lmm}{Lemma}

\usepackage{cuted}
\usepackage{enumerate}

\newcommand{\R}{\mathbb{R}}                     

\newcommand*{\QED}{\null\nobreak\hfill\ensuremath{\blacksquare}}%

\newcommand\reals[0]{\mathbb{R}}
\newcommand\sob[2]{W^{1,#1}_#2}

\newcommand\LP[2]{L^{#1}_#2}
\renewcommand\AE{\mathrm{a.e.}}
\DeclareMathOperator*{\esssup}{ess\,sup}
\newcommand\domain[0]{[0,T]}

\newcommand\mc[0]{\mathcal}
\newcommand\veps[0]{\varepsilon}

\usepackage[deletedmarkup=sout,authormarkup=superscript]{changes}
\usepackage[dvipsnames]{xcolor}
\definechangesauthor[name={Marco}, color=NavyBlue]{MN}
\definechangesauthor[name={Dominic}, color=RedViolet]{DLM}

\newif\ifarxiv
\arxivtrue

\begin{document}
\begin{frontmatter}

\title{A Semi-smooth Newton Method for the Constrained Optimal Control of Continuous-Time Linear Systems\thanksref{footnoteinfo}}  

\thanks[footnoteinfo]{This material is based upon work supported by the National Science Foundation (Grant \#2046212) and the Natural Sciences and Engineering Research Council of Canada (Reference \#\,RGPIN-2023-03257).
Corresponding author: marco.nicotra@colorado.edu}

\author[CU]{Simon J. Jones} 
\author[UBC]{Dominic Liao-McPherson}
\author[CU]{Marco~M.~Nicotra}

\address[CU]{University of Colorado Boulder, Department of Electrical, Computer, and Energy Engineering,
   Boulder, CO (USA), 80309.}
\address[UBC]{University of British Columbia, Department of Mechanical Engineering, Vancouver, BC (Canada), V6T 1Z4.}

\begin{abstract}
This paper details a novel indirect method for solving constrained optimal control problems
(OCPs) directly in continuous-time function space. The KKT conditions are embedded in a a non-smooth complementarity function, which enables their reformulation as a rootfinding problem in Banach space. This problem is then solved using a non-smooth Newton method. Finally, the paper shows that the Newton update can be obtained by solving a modified
differential Riccati equation, where the cost terms are reweighted at every iteration based on the constraint multipliers. Numerical simulations show the effectiveness of the method, which converges superlinearly up to the tolerance of the ODE solver.
\end{abstract}


\end{frontmatter}

\section{Introduction}
Continuous-time optimal control problems (OCPs) are ubiquitous across a variety of engineering domains. They are used to design and control dynamic systems in the areas of e.g., aerospace, \citep{dearing2022}, robotics \cite{bordalba2022direct}, energy, manufacturing, economics \citep{chow1976control}, quantum systems \citep{shao2024}, and many more.

Existing OCP solvers can be grouped into two main categories: ``direct'' and ``indirect'' approaches. Direct (discretize then differentiate) methods finitely parameterize the trajectories of the underlying dynamical system to transform the OCP into a nonlinear program which can then we solved using standard mathematical programming algorithms and software. There are a variety of ``direct'' approaches including pseudospectral methods \citep{gong2008spectral}, direct transcription \citep{rao2014trajectory}, multiple-shooting \citep{bock1984multiple}, and collocation,  \citep{kelly2017introduction}. These methods are well-developed but require carefully balancing tractability with accuracy when choosing how to discretize the continuous-time problem.

Indirect (differentiate then discretize) methods replace the OCP with the necessary conditions for optimality using Pontryagin's Maximum Principle \citep{kopp1962pontryagin} before solving the resulting differential-algebraic two-point boundary value problem \citep{bryson2018applied}. Indirect methods are less popular than direct approaches due to the difficulty of solving the necessary conditions numerically but are conceptually attractive for infinite horizon settings or problems where time is explicitly part of the optimization problem (e.g., flexible end times, hybrid systems, adaptive discretizations etc.) which can be challenging for direct approaches. 

Researchers have proposed a variety of indirect algorithms based on semi-smooth Newton \citep{gerdts2008global,chen2011numerical}, interior point \citep{schiela2011interior,schiela2011barrier,schiela2009barrier}, sequential quadratic programming \citep{machielsen1988numerical} and active-set \citep{hintermuller2003semi} methods. These algorithms are largely conceptual, focusing on obtaining and demonstrating convergence guarantees rather than implementation and application. The more application focused PRONTO method \citep{hauser2002} solves unconstrained OCPs using Newton's method by implementing the underlying Newton updates using differential Riccati recursion. The approach was later updated in \citep{hauser2006} to handle constraints using a primal barrier method.

However, primal-barrier methods are far from the state-of-the-art in nonlinear programming, which motivates this work which merges the PRONTO framework \citep{hauser2002} with the theoretically well-developed semi-smooth Newton based algorithms \citep{gerdts2008global,chen2011numerical} for finite-horizon constrained linear-quadratic optimal control problems. Our main contribution is showing that the two-point boundary value problem at the heart of \citep{gerdts2008global,chen2011numerical} can be solved using differential Riccati equations. This insight enables the usage of highly efficient and widely available ordinary differential equations (ODE) solvers to implement the algorithm resulting in an efficient and robust indirect method; a fact we demonstrate using numerical examples.

\textit{Notation:} Given a scalar function $f:\R^n\to\R$, let $\nabla f:\R^n\to\R^n$ be its gradient and let $\nabla^2f:\R^n\to\R^{n\times n}$ be its Hessian. Let $h:[0,T]\to\R^n$ be a measurable function defined ``almost everywhere'' (a.e.) in its domain $[0,T]$ and let 
\begin{equation}
    \|h\|_\infty = \esssup_{t\in \domain}~\|h(t)\|.
\end{equation}
The function $h$ belongs to the Lebesgue space $L_n^\infty$ if $\|h\|_\infty<\infty$. If $h(t)$ is a.e. differentiable with derivative $\dot h\!:\![0,T]\!\to\!\R^n$, the function $h$ belongs to the Sobolev space $\sob{\infty}{n}$ if $\max\{\|h\|_\infty,\|\dot h\|_\infty\}<\infty$. Given $h\in L_n^\infty$ the $\mc L_2$ norm is
\begin{equation}
    \|h\|_2^2=\int_0^Th(t)^\top h(t) dt.
\end{equation}
The space of linear operators between the Banach spaces $Z$ and $Y$ is denoted by $L(Z,Y)$. Let $\|\cdot\|_Z$ and $\|\cdot\|_Y$ be norms for the spaces $Z$ and $Y$; the induced norm of the linear mapping $F\in L(Z,Y)$ satisfies $\|F(z)\|_Y\leq \|F\|_{L(Z,Y)}~\|z\|_Z$ 

\section{Problem Setting}
The objective of this paper is to solve constrained optimal control problems in the form
\begin{subequations}\label{eq:OCP}
\begin{align}
\min_{u,x} & ~~J(x(T))
+ \int_0^T 
\ell(x(t),u(t)) dt \!\!\!\!\!\!\!\!\!\!\!\!\!\!\!\! \label{eq:OCP_cost}\\
\rm{s.t.} \ \ &   u\in\LP\infty m,~~x\in\sob\infty n \\
&x(0)=x_0,\\
& \dot{x}(t) = A x(t) + B u(t),& \AE \label{eq:OCP_dyn}\\            
              & c(x(t),u(t)) \leq 0, & \AE \label{eq:OCP_cstr}\\
              \notag
\end{align}
\end{subequations}
where $u \in \LP{\infty}{m}$ is the input signal and $x\in\sob{\infty}{n}$ is the state trajectory. We make the following assumptions to ensure the optimization problem is well-posed\medskip

\begin{ass}\label{ass:dynamics} \label{ass:costs} \label{ass:constraints}
The optimal control problem satisfies the following conditions:
\begin{enumerate}[(i)]
    \item At least one minimizer $(\bar x, \bar u) \in \sob\infty n \!\times\! \LP\infty m$ exists;
    \item The terminal cost $J:\R^n\to\R$ and incremental cost $\ell\!:\R^n\times\R^m\to\R$  are twice continuously differentiable;
    \item The terminal cost $J$ is convex and the incremental cost $\ell$ is strongly convex;
    \item Each element of $c:\R^n\times\R^m\to\R^p$ is convex and twice continuously differentiable;
    \item The pair $(A,B)$ is stabilizable;
    \item The constraints and dynamics satisfy the \citep{malanowski2003normality} assumptions of Linear Independence (A1) and Controllability (A2). 
\end{enumerate}
\end{ass}

Assumptions~\ref{ass:constraints}.(i) - \ref{ass:constraints}.(iv) are typical convexity assumptions for linear-quadratic problems and Assumption~\ref{ass:constraints}.(v) ensures that the Riccati equations associated with \eqref{eq:OCP} are solvable. All are easy to verify a-priori. Assumption~\ref{ass:constraints}.(vi) are constraint qualifications that ensure the dual variables associated with \eqref{eq:OCP_dyn} and \eqref{eq:OCP_cstr} are well-behaved and belong to $\sob\infty n$ and $\LP\infty p$ respectively. The actual conditions are technical and difficult to check, but boil down to requiring that the constraints are feasible and consistent.

\begin{remark}
    If the pair $(A,\nabla^2_{xx}\ell(x,u))$ is detectable $\forall x,u$, the assumption on the incremental cost $\ell$ can be relaxed to ``convex in $x$ and strongly convex in $u$''.
\end{remark}


\section{Solver Design}\label{sec:solver}
Our objective is to derive an indirect algorithm for solving the continuous time OCP \eqref{eq:OCP}. To this end, consider the Hamiltonian $H : \reals^n \times \reals^m \times \reals^n \times \reals^p \to \reals$
\begin{equation}\label{eq:Hamiltonian}
    H(x,u,\lambda,\mu) = \ell(x,u) + \lambda^{\!\top}\! (Ax + Bu) + \mu^{\!\top}\! c(x,u),
\end{equation}
where $\lambda\!\in\!\sob{\infty}{n}$ is the co-state trajectory and the signal $\mu\!\in\!\LP{\infty}{p}$ is the dual variable associated with the inequality constraints. Under Assumption \ref{ass:costs}, the KKT conditions
\begin{subequations}\label{eq:KKT}
\begin{align}
x(0) &=x_0,\\
\lambda(T)&= \nabla_x J(x(T)),\\
\dot{x}(t)&= A x(t) + B u(t), & \AE\label{eq:dynamics}\\
-\dot{\lambda}(t)&= \nabla_x H(x(t),u(t),\lambda(t),\mu(t)), &\AE\\
0&=\nabla_uH(x(t),u(t),\lambda(t),\mu(t)), &\AE\\
 c(x(t),u(t))&\leq 0, & \AE \label{eq:comp1}\\
\mu(t)&\geq0, & \AE \\
0 &=\mu(t)^\top c(x(t),u(t)),&\AE\label{eq:comp3}
\end{align}
\end{subequations}
are necessary optimality conditions for \eqref{eq:OCP}, as shown in [Theorem 4.3]~\citep{malanowski2003normality}.

Our goal is to design an algorithm that ``solves'' the KKT conditions \eqref{eq:KKT} using standard ODE solvers. We begin by following the same approach as \citep{gerdts2008global}: i.e., use a nonlinear complimentarity function to replace the complimentarity conditions \eqref{eq:comp1} - \eqref{eq:comp3} and transform \eqref{eq:KKT} into an infinite dimensionsal non-smooth rootfinding problem\footnote{Our choice of a semi-smooth reformulation (over e.g., an interior-point or active set approach) is motivated by the success of these methods for discrete-time OCPs e.g., \citep{dom2018,liao2020fbstab,lowenstein2024qpalm}.}.

A non-linear complimentarity problem (NCP) function $\phi:\R^2\to\R$ is a continuous function that satisfies the implication
\begin{equation} \label{eq: embedding}
    \phi(a,b)=0 \iff  a,b \geq 0,~ ab = 0.
\end{equation}
As detailed in \citep{sun1999ncp}, there are many NCP functions in the literature. In this paper, we use the ``min'' NCP function $\phi(a,b) = \min(a,b)$ and the Fischer-Burmeister Function $\phi(a,b) = a + b - \sqrt{a^2 + b^2}$. These functions are not continuously differentiable, but are \textit{semismooth}, \citep{mifflin1977semismooth}, which means they are Lipschitz continuous with well-behaved generalized gradients.

Applying an NCP function element-wise to the KKT conditions \eqref{eq:KKT}, we obtain the equivalent conditions
\begin{subequations} \label{eq:pmp_nl}
\begin{align}
x(0)-x_0 &=0,\qquad&\label{eq:x0}\\
\lambda(T)-\nabla_x J(x(T))&=0,\label{eq:lT}\\
\dot{x}(t)-Ax(t)-Bu(t)&=0,& \AE\label{eq:dyn}\\
\dot{\lambda}(t)+\nabla_x H(x(t),u(t),\lambda(t),\mu(t))&=0 , &\AE\label{eq:Dl}\\
\nabla_u H(x(t),u(t),\lambda(t),\mu(t))&=0,&\AE\label{eq:Du}\\
\phi(\mu(t),-c(x(t),u(t)))&=0,&\AE\label{eq:phi}
\end{align}
\end{subequations}
Given $z = (x,u,\lambda,\mu)$, these can be written compactly as
\begin{equation}\label{eq:F=0}
    F(z) = 0,
\end{equation}
where $F:Z \to Y$ is a mapping between the Banach spaces
\begin{equation}
    Z = \sob{\infty}{n} \times \LP{\infty}{m} \times \sob{\infty}{n} \times \LP{\infty}{p}
\end{equation}
and
\begin{equation}
    Y = \reals^n \times \reals^n \times \LP{\infty}{n} \times \LP{\infty}{n} \times \LP{\infty}{m} \times  \LP{\infty}{p}.
\end{equation}


The rootfinding problem \eqref{eq:F=0} can be solved using the \citep{ulbrich2002semismooth} semi-smooth variant of Newton's method, which results in an iterative process
\begin{subequations}\label{eq:Newton}
\begin{equation} 
    z^+ =z + \tilde z, 
\end{equation}
where the update $\tilde z$ is obtained by solving
\begin{equation}\label{eq:gen_Newt}
     F(z) + G \tilde z = 0,\qquad G \in \partial F(z)
\end{equation}
\end{subequations}
and $\partial F: Z \rightrightarrows L(Z,Y)$ is the Clarke generalized Jacobian \citep[Definition 3.35]{ulbrich2002nonsmooth} of $F$. This iteration converges to a solution of the KKT conditions \eqref{eq:KKT} under mild conditions. 
\begin{thm}\label{thm:Convergence}\citep[Theorem 2.6]{gerdts2008global}
Let Assumption~\ref{ass:dynamics} hold and let $\bar z\in Z$ satisfy $F(\bar z) = 0$. Further, suppose that there exist $\veps > 0$ and $\delta > 0$ such that, if $\|z - \bar z\|_Z \leq \veps$, all $G \in \partial F(z)$ are non-singular with $\|G^{-1}\|_{L(Y,Z)} \leq \delta$. Then, the Newton iteration \eqref{eq:Newton} converges to $\bar z$ superlinearly.
\end{thm}

\begin{remark}
As noted in \citep{gerdts2008global}, it is straightforward to globalize \eqref{eq:Newton} using $ \|F(z)\|_2^2$ as a merit functions.
\end{remark}


\subsection{Solving for Newton-steps using a Riccati equation}
The key step in developing an practical Newton-type method is proposing a reliable and efficient approach for solving the Newton-step equation \eqref{eq:gen_Newt}. In this section, we provide a Riccati-based method that can be readily implemented using numerical integration tools.

Specializing the semi-smooth Newton update \eqref{eq:Newton} to the Banach space mapping \eqref{eq:pmp_nl} yields an affine two-point boundary value problem with boundary conditions
\begin{subequations} \label{eq:newton_step_expanded}
\begin{align}
\tilde x(0) &= r_1,\\
\tilde \lambda(T) - \nabla_x^2J(x(T)) \tilde x(T) &= r_2,
\end{align}
and differential-algebraic equations\footnote{We have used operator notation to shorten the otherwise unwieldy equations, e.g., $H''_{xx} \tilde x \iff \nabla^2_{xx}H(z(t)) \tilde x(t)$ almost everywhere.}
\begin{gather}
 \dot{\tilde x} = A \tilde x + B \tilde x + r_3,\qquad\qquad\qquad\qquad~\\
-\dot{\tilde\lambda} = H''_{xx}\tilde x + H''_{xu} \tilde u + H''_{x\lambda} \lambda + H''_{x\mu}\tilde \mu + r_4,\\
H''_{ux}\tilde x+ H''_{uu}\tilde u + H''_{u\lambda}\lambda + H''_{u\mu}\tilde \mu = r_5,\\
\qquad\qquad\partial_x \phi \tilde x +\partial_u \phi \tilde u+ \partial_\mu \phi \tilde\mu = r_6, \label{eq:KKT_mu}
\end{gather}
\end{subequations}
where the residuals follow directly from \ref{eq:pmp_nl}
\begin{equation}\label{eq:residuals}
     \begin{bmatrix}
    r_1\\ r_2\\ r_3(t) \\r_4(t) \\r_5(t)\\ r_6(t)
\end{bmatrix} = \begin{bmatrix}
    x_0-x(0)\\
    \nabla_x J(x(T))-\lambda(T)\\
    Ax(t)+Bu(t)-\dot{x}(t)\\
    \nabla_x H(z(t))+\dot{\lambda}(t)\\
    -\nabla_uH(z(t))\\
    -\phi(z(t))
\end{bmatrix}.
\end{equation}

The first step in the derivation is to eliminate the dependency on the Lagrange multipliers $\tilde\mu$. The generalized derivatives in \eqref{eq:KKT_mu} can be expanded into the elementwise equations
\begin{equation} \label{eq:comp-deriv}
    \gamma_i(t) \left(\nabla_x^\top\!c_i(t) \tilde x + \nabla_u^\top\!c_i(t)
    \tilde u\right)+ \eta_i(t)\tilde\mu = -\phi_i(z)
\end{equation}
where 
\begin{equation}
     (\eta_i(t),\gamma_i(t)) \in \partial \phi_i(\mu_i(t),-c_i(x(t),u(t)))
\end{equation} 
are elements of the generalized Jacobian\footnote{For the FB function, we have
\begin{multline*}
    \partial \phi_{FB}(\mu,-c) = \begin{cases}
    \left\{\left (1\!-\!\frac{\mu}{\sqrt{\mu^2 + c^2}}, -1\!-\!\frac{c}{\sqrt{\mu^2 + c^2}} \right) \right\} & \textrm{if}~(a,b) \neq 0\\
    \{(\eta,\gamma)~|~ (1-\eta )^2 + (1+\gamma)^2 = 1\} & \textrm{if}~ (a,b) = 0,
    \end{cases}
\end{multline*}
and for the min function, we have
\begin{equation*}
    \partial \phi_{min}(\mu,-c) = \begin{cases}
    \{(1,0)\} & \textrm{if}~\mu > -c\\
    \{(0,-1)\} & \textrm{if}~ \mu < -c\\
    \{\eta\geq0,\gamma\leq0~|~\eta - \gamma = 1\} & \textrm{if}~ \mu = -c.
    \end{cases}\!\!\qquad\qquad\qquad
\end{equation*}} of the NCP function $\phi$.

Rearranging equation \eqref{eq:comp-deriv}, we can isolate $\tilde \mu$
\begin{equation}\label{eq: replace mu}
\tilde\mu_i(t)\!=\!\alpha_i(t)+\beta_i(t)(\nabla_x^\top\!c_i(t)\tilde x(t)+\nabla_u^\top\!c_i(t)\tilde u(t)),
\end{equation}
where 
\begin{equation}\label{eq: Barrier}
    \alpha_i(t)=\frac{-\phi_i(t)}{\eta_i(t) + \delta},\quad \beta_i(t)=\frac{-\gamma_i(t)}{\eta_i(t) + \delta},
\end{equation}
and $\delta > 0$ is a small regularization parameter to guard against singularity. Assuming $r_3(t)=0$ almost everywhere and noting that \eqref{eq:Hamiltonian} implies
\begin{equation}
     \displaystyle H''_{x\mu}(t)\tilde \mu(t)=\sum_{i=1}^{n_c}\nabla_x c_i(t)\tilde\mu_i(t),
\end{equation}
the Newton-step system \eqref{eq:newton_step_expanded} can be re-written as
\begin{subequations}\label{eq:LQR}
\begin{align}
\tilde x(0) &=x_0-x(0),\qquad& \label{eq: x0}\\
\lambda^{\!+}\!(T)&=\nabla_x^2 J(x(T))\tilde x(T)+\nabla_x J(x(T)),\label{eq: lT}\\
\dot{\tilde x}(t)&=A\tilde x(t)+B\tilde u(t), \label{eq:dot x}\\
-\dot\lambda^{\!+}\!(t)&=\!Q(t)\tilde x(t)\!+\! S(t)\tilde u(t)\!+\! A^\top\lambda^{\!+}\!(t)\!+\!q(t), \label{eq:dot l}\\
0&=\!S^\top\!(t)\tilde x(t)\!+\! R(t)\tilde u(t)\!+\! B^\top\! \lambda^{\!+}\!(t)\!+\!r(t),\label{eq:u}
\end{align}
\end{subequations}
with
\[
\begin{array}{rl}
    Q(t) &\!\!=\! \nabla_{xx}^2\ell(t)\!+\!\sum_{i}\mu_i(t)\nabla^2_{xx}c_i(t)+\beta_i(t)\nabla_x c_i(t)\nabla_x^\top c_i(t),\\
    R(t) &\!\!=\! \nabla_{uu}^2\ell(t)\!+\!\sum_{i}\mu_i(t)\nabla^2_{uu}c_i(t)+\beta_i(t)\nabla_u c_i(t)\nabla_u^\top c_i(t),\\
    S(t) &\!\!=\! \nabla_{xu}^2\ell(t)\!+\!\sum_{i}\mu_i(t)\nabla^2_{xu}c_i(t)+\beta_i(t)\nabla_x c_i(t)\nabla_u^\top c_i(t),\\
    q(t) &\!\!=\! \nabla_x\ell(t)+\sum_{i}(\mu_i(t)+\alpha_i(t))\nabla_xc_i(t),\\
    r(t) &\!\!=\! \nabla_u\ell(t)+\sum_{i}(\mu_i(t)+\alpha_i(t))\nabla_uc_i(t),
\end{array}
\]
where we used $\lambda^{\!+}\!=\lambda+\tilde\lambda$ to write the equation directly in terms of $\lambda^+$. Since Equation \eqref{eq:dot x} is valid only if $r_3(t)=0$, the following proposition ensures recursive feasibility if the initial guess $(x_0,u_0)$ satisfies $\dot x_0=Ax_0+Bu_0$.\medskip

\begin{prop}\label{prop: Traj}
Let $(x,u)\!\in\! W^{1,\infty}_n\!\times\! L^{\infty}_m$ satisfy $\dot x\!=\!Ax\!+\!Bu$, and let $(\tilde x,\tilde u)\in W^{1,\infty}_n\times L^{\infty}_m$ satisfy \eqref{eq:dot x}. Then, given an arbitrary scalar $\gamma\in(0,1]$, the Newton update 
\begin{equation}
    x^+=x+\gamma\tilde x,\qquad u^+=u+\gamma\tilde u,
\end{equation}
satisfies $\dot x^+=Ax^++Bu^+$.
\end{prop}

\begin{pf}
    The result follows directly from \eqref{eq:dot x}, since
    \begin{equation}
        \dot x+\gamma \dot{\tilde x}=A(x+\gamma \tilde x)+B(u+\gamma \tilde u).
    \end{equation}
    \QED
\end{pf}

A closer inspection of \eqref{eq:LQR} reveals that they are the necessary conditions for an unconstrained time-varying linear-quadratic regulator problem
\begin{subequations}\label{eq:LQR-OCP}
\begin{align}
\min_{\tilde x,\tilde u} & ~~ \hat J(\tilde x(T)) + \int_0^T  \hat\ell(\tilde x(t),\tilde u(t),t) dt \!\!\!\!\!\!\!\!\!\!\!\!\!\!\!\!\\
\rm{s.t.} \ \ & \tilde x(0)= x(0) - x_0,\\
& \dot{\tilde x}(t) = A \tilde x(t) + B \tilde u(t),& \AE
\end{align}
\end{subequations}
with $\hat J(\tilde x) = \hat \tilde x^\top  P \tilde x + \tilde x^\top  p$ and 
\begin{multline}
     \hat\ell(\tilde x(t),\tilde u(t),t) = \\ \begin{bmatrix}\tilde x(t)\\\tilde u(t) \end{bmatrix}^\top
\begin{bmatrix}
    Q(t) & S^\top (t)\\
    S(t) & R(t)
\end{bmatrix}
\begin{bmatrix}\tilde x(t)\\\tilde u(t)  \end{bmatrix}
+ \begin{bmatrix}\tilde x(t)\\\tilde u(t) \end{bmatrix}^\top \begin{bmatrix}
    q(t)\\ r(t)
\end{bmatrix},
\end{multline}
where the inequality constraints have been incorporated into the problem in the form of ``barrier'' cost terms $\sum_i (\mu_i + \alpha_i) \nabla c_i$ and $\sum_i \mu_i \nabla^2 c_i + \beta_i (\nabla c_i) (\nabla c_i)^T$. Newton's method can thus be interpreted as an adaptive search for the linear and quadratic weights of the cost function.

The main advantage of \eqref{eq:LQR-OCP} is that it can be solved analytically. From \eqref{eq:u} we obtain
\begin{equation}\label{eq: replace u}
    \tilde u(t) = -R(t)^{-1}(r(t)+B(t)^\top\lambda^{\!+}\!+S(t)^\top\!\tilde x(t)).
\end{equation} 
To address the costates, consider the affine mapping
\begin{equation}\label{eq: affine map}
    \lambda^+(t) = P(t)\tilde x(t)+p(t),
\end{equation}
and its time derivative
\begin{equation}\label{eq: l_dot}
    \dot\lambda^+(t) = \dot P(t)\tilde x(t)+P(t)\dot{\tilde x}(t)+\dot p(t).
\end{equation}
By replacing \eqref{eq:dot x}, \eqref{eq:dot l}, and \eqref{eq: replace u} into \eqref{eq: l_dot}, and factoring out $\tilde x$, we obtain the differential equations
\begin{subequations}\label{eq:Riccati}
\begin{align}
-\dot P &= Q+A^\top\! P\!+\!PA-(PB\!+\!S)R^{-1}(B^\top\!P\!+\!S^\top\!), \label{eq:Riccati1}\\
-\dot p &= q+A^\top\! p-(PB\!+\!S)R^{-1}(B^\top\! p+r),\label{eq:Riccati2}\\
P&(T)=J''_{xx}(x(T)),\\
p&(T)=\nabla_x J(x(T)),
\end{align}
\end{subequations}
which can be solved backwards in time. The state trajectory can then be obtained by solving
\begin{subequations}\label{eq: Optimal Traj}
\begin{align}
    \dot{\tilde x}&=(A-BR^{-1}(B^\top\!P+S))\tilde x-BR^{-1}(B^\top\!p+r),\\
    \tilde x&(0)=x(0)-x_0,
\end{align}
\end{subequations}
forward in time. Once $\tilde x(t)$ is known, the remaining signals $\tilde u(t)$ and $\lambda^{\!+}\!(t)$ can be obtained from \eqref{eq: replace u} and \eqref{eq: affine map}, respectively.



\subsection{Stopping Condition}\label{ssec:Stop}
Since the problem \eqref{eq:KKT} is infinite dimensional, the choice of stopping criterion is a slightly more involved than in the finite dimensional case. For this method, we use 
\begin{equation} \label{eq:stopping}
    \|F(z)\|_2 \leq \epsilon_t
\end{equation}
as we found using a $2$ norm (instead of e.g., $\|F(z)\|_\infty$) is more numerically stable. We will show in Section~\ref{ss:analysis} that this leads to a well-defined stopping condition.


Given the residual vector partitioned in \eqref{eq:residuals}, our approach is designed to automatically enforce $r_1=0$, $r_2=0$, and $r_3(t)=0$ at every iteration. Therefore, $\|F(z)\|_2 =\sqrt{\|r_4(z)\|_2^2+\|r_5(z)\|_2^2+\|r_6(z)\|_2^2}$. By replacing \eqref{eq: l_dot} and \eqref{eq:Riccati1} in \eqref{eq:dot l}, we obtain
\begin{equation*}
    \|r_4\|_2^2 ~=\int_0^T  \!\!\!\left\|Q(t)\tilde x(t)+S(t)\tilde u(t)+{\textstyle\sum_i}\alpha_i(t)\nabla_xc_i(t)\right\|^2dt.
\end{equation*}
Likewise,
\begin{align*}
\|r_5\|_2^2 ~=\displaystyle\int_0^T &\|\nabla_u\ell(t)+B^\top\lambda(t)+{\textstyle\sum_i}\mu_i(t)\nabla_u c(t)\|^2dt,\\
\|r_6\|_2^2 ~=\displaystyle\int_0^T &\|\phi(t)\|^2dt.
\end{align*}
which can be easily computed via quadrature. 

\subsection{Implementation}
Our proposed algorithm is summarized in Algorithm \ref{alg:newton}. The core of the algorithm is a semi-smooth Newton method \eqref{eq:Newton} with a Riccati-based solver. The main computational burden of the algorithm is solving initial value problems and evaluating 1D integrals numerically, for which there are plenty of mature and widely available software packages, e.g., \citep{hindmarsh2005sundials}.

\begin{algorithm}[ht] 
  \caption{Semi-smooth Newton Step}
  \label{alg:newton}
  \begin{algorithmic}[1]
  \Require
      \Statex Problem Setting: $x_0$, $T$, $J(x)$, $\ell(x,u)$, $c(x,u)$, $A$, $B$
      \Statex Gradients/Hessians of: $J(x)$, $\ell(x,u)$, $c(x,u)$
      \Statex NCP function: $\phi(a,b)$, $\eta(a,b)$, $\gamma(a,b)$
      \Statex Solver parameters: $\delta$, $\epsilon_t$
      \Statex 
\Statex \hspace{-20pt} \textbf{Main:}
      \State $[x_0,u_0,\lambda_0,\mu_0]=$\textsc{Initialize}
    \While{\textsc{Residual}$(x_0,u_0,\lambda_0,\mu_0)\geq \epsilon_t$}
      \State $[\tilde x,\tilde u,\tilde\lambda,\tilde\mu]=$\textsc{NewtonUpdate}$(x,u,\lambda,\mu)$
      \State $[x_0, u_0 ,\lambda_0,\mu_0]\leftarrow[x_0, u_0 ,\lambda_0,\mu_0]+[\tilde x, \tilde u ,\tilde \lambda,\tilde \mu]$
      \EndWhile
    \Statex
\Statex \hspace{-20pt} \textbf{Functions:}
\Function{$[x_0,u_0,\lambda_0,\mu_0]=$Initialize}{}
      \State Use \eqref{eq:LQR} to compute $[Q,R,S,q,r]$, with $\mu=0$  
      \State Integrate \eqref{eq:Riccati} backwards in time to find $[P,p]$ 
      \State Integrate \eqref{eq: Optimal Traj} forward in time to find $\tilde x$
      \State Use \eqref{eq: replace u}-\eqref{eq: affine map} to compute $[\tilde u,\lambda^+]$
      \State \textbf{return} $[\tilde x,\tilde u,\lambda^+,0]$
    \EndFunction
      \Statex
    \Function{$[\tilde x,\tilde u,\tilde\lambda,\tilde\mu]=$NewtonUpdate}{$x,u,\lambda,\mu$}
      \State Use \eqref{eq: Barrier}, \eqref{eq:LQR} to compute $[Q,R,S,q,r]$
      \State Integrate \eqref{eq:Riccati} backwards in time to find $[P,p]$
      \State Integrate \eqref{eq: Optimal Traj} forward in time to find $\tilde x$
      \State Use \eqref{eq: replace u}-\eqref{eq: affine map} and \eqref{eq: replace mu} to compute $[\tilde u,\lambda^+,\tilde\mu]$
      \State \textbf{return} $[\tilde x,\tilde u,\lambda^+\!-\!\lambda,\tilde\mu]$
    \EndFunction
      \Statex
      \Function{Residual}{$x,u,\lambda,\mu$}
      \State Use \eqref{eq: Barrier}, \eqref{eq:LQR} to compute $[Q,R,S,\alpha]$
      \State Use Subsection \ref{ssec:Stop} to compute $\|r_4\|_2^2$, $\|r_5\|_2^2$, $\|r_6\|_2^2$
      \State \textbf{return} $\|r_4\|_2^2+\|r_5\|_2^2+\|r_6\|_2^2$
    \EndFunction
  \end{algorithmic}
\end{algorithm}

\section{Analysis} \label{ss:analysis}
In this section, we show that Algorithm~\ref{alg:newton} is well-defined, terminates finitely, and generates a sequence that satisfies the first-order necessary conditions \eqref{eq:KKT} of the optimal control problem \eqref{eq:OCP}.

\begin{thm}
Let Assumption~\ref{ass:dynamics} hold, let $\bar z\in Z$ satisfy the KKT conditions \eqref{eq:KKT}, and let $\{z_k\} = \{(x_k,u_k,\lambda_k,\mu_k)\} \subset Z$ denote the sequence generated by Algorithm~\ref{alg:newton} starting from $z_0 = (x_0,u_0,\lambda_0,\mu_0)\in Z$. Then there exists $\epsilon > 0$ such that if $\|z_0 - \bar z\|_\infty \leq \epsilon$ then:
\begin{enumerate}[i)]
    \item the sequence $\{z_k\}$ is well-defined;
    \item $\|z_k -\bar z\|_\infty \to 0$ at a superlinear rate;
    \item the stopping condition $\|F(z_k)\|_2 \leq \epsilon_t$ is satisfied in finite time.
\end{enumerate}
\end{thm}
\begin{pf}
\ifarxiv
See Section \ref{sec:appendix} for proof. \QED
\else
See Appendix in \cite{arxiv} for proof. \QED
\fi
\end{pf}
The result demonstrates local superlinear convergence under the regularity conditions in Assumption~\ref{ass:dynamics} (strong convexity and constraint qualifications) and is a typical result for Newton-type methods.

\section{Simulation}
We apply Algorithm \ref{alg:newton} to the emergency lane change maneuver in \cite{skibik2022terminal}. The state of the system is $x^\top = [s ~
\phi ~ \beta ~ \omega~ \delta_f]^\top$, with $s$ being the lateral position of the
vehicle, $\phi$ the yaw angle, $\beta = \dot{s}/V_x$ the sideslip angle,
$\omega = \dot{\phi}$ the yaw rate, and $\delta_f$ is the front wheel steering angle. The control input $u=\dot\delta_f$ is the steering rate. Given a constant speed of
$V_x = 30\ \rm{m}/ \rm{s}$, the linearized continuous-time dynamics are
\[
A = \begin{bmatrix}
0 & V_x & V_x & 0 & 0 \\
0 & 0   & 0   & 1 & 0\\
0 & 0   & -\frac{2 C_\alpha}{m V_x} & \frac{C_\alpha (\ell_r - \ell_f)}{m V_x^2} - 1 &0\\
0 & 0   & \frac{C_\alpha(\ell_r - \ell_f)}{I_{xx}} & -\frac{C_\alpha (\ell_r^2 + \ell_f^2)}{I_{xx} V_x}&0\\
0&0&0&0&0
\end{bmatrix}
\quad B = \begin{bmatrix}
    0 \\
    0 \\
    \frac{C_\alpha}{m V_x} \\ \frac{C_\alpha \ell_f}{I_{xx}}\\1
\end{bmatrix},
\]
where $\ell_f = 1.56 \ \rm{m}$ and $\ell_r = 1.64 \ \rm{m}$ are the moment arms
of the front and rear wheels relative to the center of mass, $C_\alpha = 246994
\ \rm{N}/\rm{rad}$ is the tire cornering stiffness, $m=2041 \ \rm{kg}$ is the
mass of the vehicle, and $I_{zz} = 4964 \ \rm{kg} \cdot \rm{m}^2$ is the moment
of inertia about the yaw axis. These parameters approximately represent a 2017 BMW 740i sedan. The system is subject to output constraints on $y^\top = [
\alpha_f \quad \alpha_r \quad \delta_f ]^\top$, where $\alpha_f, \alpha_r$ are
the front and rear sideslip angles
\begin{subequations}
\begin{align}
\alpha_f=&\  \delta_f - \arctan\left(\beta + \frac{\ell_f}{V_x} \omega\right), \\
\alpha_r =&
  -\arctan(\beta - \frac{\ell_r}{V_x} \omega).
\end{align}
\end{subequations}
Specifically, the system must satisfy $y\in\mathcal Y$, with
\begin{equation}
\mathcal{Y} = [-8^\circ, 8^\circ] \times [-8^\circ, 8^\circ] \times [-30^\circ, 30^\circ],
\end{equation}
which limits the front and rear slip angles to prevent drifting, and accounts for mechanical limits on the steering angle.
The incremental cost and terminal cost are, respectively,  $\ell(x,u)=\frac12x^\top Qx+\frac12 u^\top Ru$ and $J(x) = \frac12x^\top P x$, where $Q = \rm{diag}(1, 1, 0 ,0, 0)$, $R = 0.1$, $S = 0_{5\times 1}$, and $P$ is given by the solution of the discrete algebraic Riccati
equation.  The chosen time horizon is $T=4\textrm{s}$.

Figure~\ref{fig:states-00-10-30-90} shows the solution estimates of the lateral position $s$ and the constrained outputs $\delta_f$, $\alpha_f$, and $\alpha_r$ at iterations $0,\ 10,\ 30,$ and $90$. Figure~\ref{fig:residuals} shows the 2-norm of the residual vector $F(z)$ at every iteration using both the $\min$ and the Fischer-Burmeister NCP functions. Although the $\min$ function begins at a higher value, the two display fairly similar behaviors. The 2-norm of the residual vector decreases by more than 4 orders of magnitude in the first 10 iterations and then tends to an asymptote as the ODE solver approaches machine precision.

\begin{figure}
    \centering
    \includegraphics[width=\linewidth]{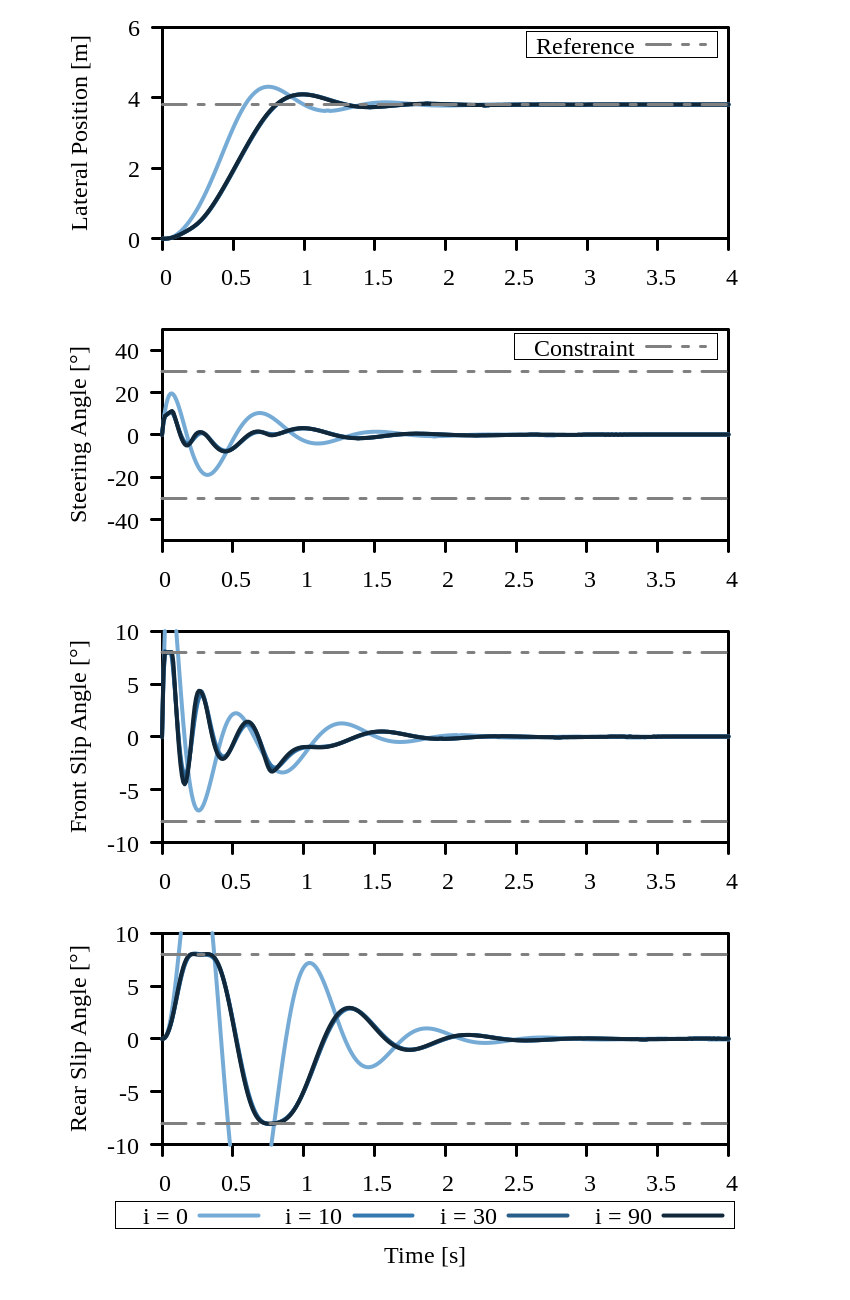}
    \caption{Result of Newton-step solver at iterations $i = 0,\ 10,\ 30,$ and $90$ using Fischer-Burmeister as NCP function.}
    \label{fig:states-00-10-30-90}
\end{figure}

\begin{figure}
    \centering
    \includegraphics[width=0.9\linewidth]{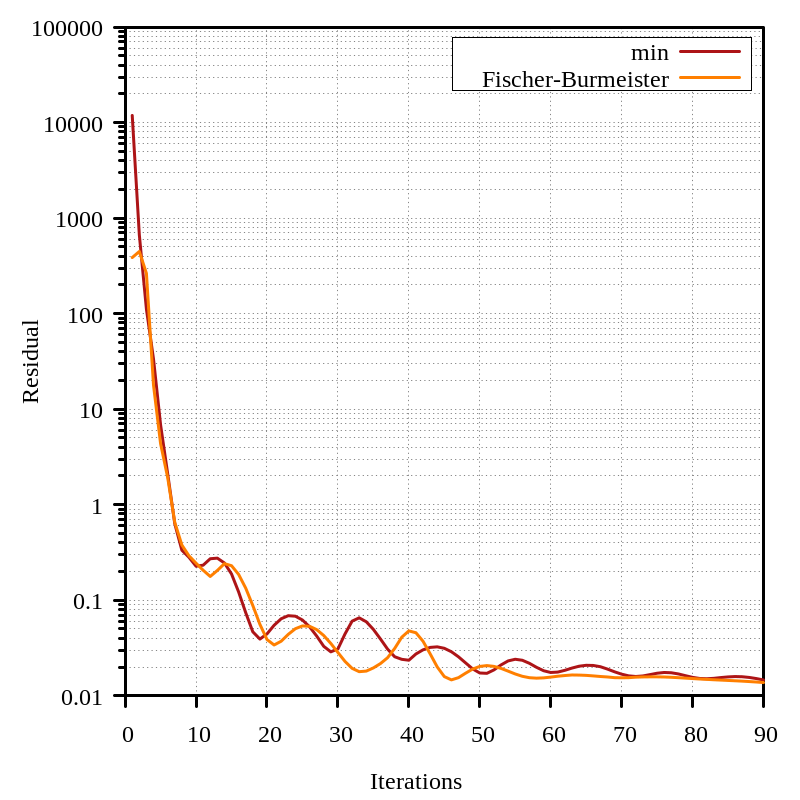}
    \caption{$||F(z)||_2$ versus number of solver iterations for both $\min$ and Fischer-Burmeister NCP functions. Regularization parameter $\delta$ was chosen experimentally. For min function: $\delta = 1\times10^{-1}$. For Fischer-Burmeister function: $\delta = 1\times 10^{-2}$}.
    \label{fig:residuals}
\end{figure}

\section{Conclusions}

In this work, we defined a novel method for solving a constrained OCP by using an NCP function to transform the KKT necessary conditions into a non-smooth rootfinding problem in function space. By applying a semi-smooth variant of Newton's method, we showed that the Newton updates could be computed by solving a differential Riccati equation at each solution step. Numerical simulations demonstrated the solver's convergence to an optimal solution. Future research efforts will extend the proposed method to nonlinear systems.

\bibliography{Refs}       

\ifarxiv
\section{Appendix}\label{sec:appendix}
\subsection{Proof of Theorem 3}\label{sec:proof-of-thm3}
We begin by showing that the operator $G$ in \eqref{eq:newton_step_expanded} is invertible and bounded in a neighborhood of $\bar z$. Invertibility then immediately implies that $\{z_k\}$ is well-defined.

\begin{lmm}\label{thm:nonsingular}Consider the Newton-step equation \eqref{eq:newton_step_expanded}, let $\phi:\R^2\to\R$ be an NCP function, and let $(\eta,\gamma)\in[0,1]\times[0,1]$ be elements of its generalized Jacobian. Then, under Assumption \ref{ass:constraints}, there exists $\epsilon > 0$ and $M\in (0,\infty)$ such that $\|z - \bar z\| \leq \epsilon$ implies $\|G^{-1}\|_{L(Z,Y)} \leq M$ for all $G\in \partial F(z)$.

\end{lmm}

\begin{pf}
Given $z\in Z$, the assumptions on $\phi$ ensure that \eqref{eq: Barrier} satisfies $(\alpha,\beta)\!\in\!\LP{\infty}{p}\!\times\!\LP{\infty}{p}$ and $\beta(t)\geq0$. Since Assumptions \ref{ass:constraints}.(iii-v) imply $\nabla^2\ell(t)\!\succ\!0$ and $\nabla^2 c_i(t)\!\succeq\!0$, and since the solution to \eqref{eq:KKT} must satisfy $\bar\mu(t)\geq0$, there exists a neighborhood of $\bar z$ such that 
\begin{equation}\label{eq:suff_conds}
\nabla^2\ell(t)+\sum_{i=1}^p\mu_i(t)\nabla^2c_i(t)+\beta_i(t)\nabla c_i(t)\nabla^\top\! c_i(t)\succ0.
\end{equation}
This is sufficient to ensure 
\begin{equation}\label{eq:Schur}
    \qquad\qquad\qquad\begin{bmatrix}
        Q(t)&S^\top(t)\\S(t) &R(t)
    \end{bmatrix}\succ0,\qquad\qquad\AE
\end{equation}
Combining this property with Assumption \ref{ass:constraints}.(ii) ensures that the solution to the differential Riccati equation \eqref{eq:Riccati1} exists and is unique\footnote{The existence and uniqueness of the solution also holds under the weakened conditions $Q(t)\succeq0$ and $(A,Q(t))$ detectable $\forall t\in[0,T]$.}. Since \eqref{eq:Schur} implies $R(t)\succ0$, \eqref{eq:Riccati2} and \eqref{eq: Optimal Traj} are globally Lipschitz continuous, which implies that their solutions also exist and are unique. This is sufficient to show $\tilde x\in\sob\infty n$. Moreover, \eqref{eq: affine map} implies  $\lambda^+\in\sob\infty n$, \eqref{eq: replace u} implies $\tilde u\in\LP\infty m$, and \eqref{eq: replace mu} implies $\tilde\mu\in\LP\infty p$. Finally, since $\lambda\in\sob\infty n$, we obtain $\lambda^+\!\!-\!\lambda=\tilde\lambda\in\sob\infty n$.\medskip

Next we will show that  $\zeta=-G^{-1}(F(z)+\Gamma)$ satisfies $\|\zeta-\tilde z\|_\infty\leq M\|\Gamma\|_\infty,$ $\forall \Gamma\in Y$, which immediately implies $\|G^{-1}\|_{L(Z,Y)} = \sup_{\Gamma \neq 0} \|G^{-1}\Gamma\|/\Gamma \leq M$. 

To prove the statement, it is sufficient to track how a change in the residual \eqref{eq:residuals} affects the solution to the two-point boundary problem \eqref{eq:newton_step_expanded}. \\Given $r_6(t)=-\phi(z(t))+\Gamma_6(t)$, \eqref{eq:residuals} leads to a modified
\begin{equation}
    \alpha_i(t)=\frac{-\phi_i(t)+\Gamma_{6,i}(t)}{\eta_i(t)+\delta},
\end{equation}
which, in turn, affects the linear cost terms $q(t)$ and $r(t)$. Likewise, $\Gamma_4(t)$ and $\Gamma_5(t)$ translate into additive disturbances on $q(t)$ and $r(t)$, respectively. Since $q$ and $r$ do not appear in \eqref{eq:Riccati1}, $P(t)$ is unaffected. Conversely, since \eqref{eq:Riccati2} is affine in both $q$ and $r$, its solution $p(t)$ is Lipschitz continuous with respect to $\Gamma_4\in L_n^\infty$, $\Gamma_5\in L_m^\infty$, and $\Gamma_6\in L_p^\infty$. Moreover, since the residual $r_2$ maps to the boundary conditions $p(T)=\nabla_xJ(x(T))+\Gamma_2$, $p(t)$ is also Lispchitz continuous with respect to $\Gamma_2\in\R^n$.

We are now left to study the effects of $\Gamma_1\in\R^n$ and $\Gamma_3\in L_n^\infty$. The former simply changes the initial conditions $\tilde x(0)=x(0)-x_0+\Gamma_1$, whereas the latter modifies \eqref{eq: Optimal Traj} into
\[
\dot{\tilde x}=(A-BR^{-1}(B^\top\!P+S))\tilde x-BR^{-1}(B^\top\!p+r)+\Gamma_3,
\]
which is affine in $p$, $r$, and $\Gamma_3$. Combined with the Lipschitz continuity properties of $p(t)$, we show that $\tilde x(t)$ is Lipschitz continuous with respect to $\Gamma\in Y$. Lipschitz continuity of $\tilde z=(\tilde x,\tilde u,\tilde\lambda,\tilde\mu)$ then follows from \eqref{eq: replace u}-\eqref{eq: affine map} and \eqref{eq: replace mu}. \QED
\end{pf}

Using Lemma~\ref{thm:nonsingular}, we can now invoke Theorem~\ref{thm:Convergence} to conclude that $\|z_k -\bar z\|_\infty\to0$ superlinearly.

Finally, we need to show that the stopping condition $\|F(z_k)\|_2 \leq \epsilon_t$ is satisfied in finite time. To prove this, we note that, for any $f\in Z$,
\begin{align}
    \|f\|_2^2 &= \int_0^T \|f(t)\|^2 dt \leq \int_0^T \esssup_{t\in [0,T]} \|f(t)\|^2 dt \\
    &= T \esssup_{t\in [0,T]} \|f(t)\|^2 = T \|f\|_\infty^2.
\end{align}
Since $\|z_k - \bar z\|_\infty \to 0$ superlinearly, it then follows that $\|F(z_k)\|_2\leq \sqrt{T} \|F(z_k) - F(\bar z)\|_\infty \leq \sqrt{T} \Delta \|z_k - \bar z\|_\infty \to 0$ (where $\Delta$ is the Lipschitz constant of $F$) as well. Thus, the stopping condition will be satisfied in finite time.

\fi

\end{document}